\documentclass[a4paper,12pt, reqno]{amsart}
\usepackage{amsmath,amsfonts,amssymb,amsbsy,amsthm,epsfig,color,mathrsfs,nicefrac}
\usepackage{hyperref} 

\newcommand{\ignore}[1]{}

\newcommand{\supp}{\operatorname{supp}}

\newcommand{\bb}{\mathbb}

\newcommand{\C}{\bb C} 

\newcommand{\Z}{\bb Z}
\newcommand{\height}{\textrm{\rm H}}

\newcommand{\R}{\bb R}
\newcommand{\N}{\bb N}

\newcommand{\Q}{\mathbb Q}

\newcommand{\cH}{\mathcal{H}}

\newcommand{\J}{Jarn\'{i}k}

\newtheorem{Theorem}{Theorem}

\newtheorem{Prop}[Theorem]{Proposition}
\newtheorem{Lemma}[Theorem]{Lemma}
\newtheorem*{lemma*}{Lemma}

\newtheorem{example}[Theorem]{Example}

\newtheorem*{theorem*}{Theorem}
\newtheorem{Def}[Theorem]{Definition}
\numberwithin{equation}{section}
\numberwithin{Theorem}{section}

\begin{document}
\title[Diophantine approximation]{Metric Diophantine approximation on homogeneous varieties}
\author{Anish Ghosh, Alexander Gorodnik, and Amos Nevo} 
\address{School of Mathematics, University of East Anglia, Norwich, UK }
\email{a.ghosh@uea.ac.uk}
\address{School of Mathematics, University of Bristol, Bristol UK }
\email{a.gorodnik@bristol.ac.uk}
\address{Department of Mathematics, Technion IIT, Israel}
\email{anevo@tx.technion.ac.il}

\date{\today}

\begin{abstract}
We develop the metric theory of Diophantine approximation on homogeneous varieties
of semisimple algebraic groups and prove results analogous to the classical Khinchin and Jarnik theorems.
In full generality our results establish  simultaneous Diophantine approximation with respect
to several completions, and Diophantine approximation over general number fields using 
$S$-algebraic integers. In several important examples, the metric results we obtain are optimal. 
The proof uses quantitative equidistribution properties of suitable averaging operators, which are derived from 
 spectral bounds in automorphic representations.
\end{abstract}

\maketitle


\section{Introduction}

In the classical theory of Diophantine approximation one studies
approximations of vectors $x\in \mathbb{R}^d$ by rational vectors. 
Beginning with works of Khintchine in the 1920s, there has been a rich literature investigating
the size of the sets of vectors in $\R^d$ satisfying various approximation properties \cite{sp,h,BD}.
To set the scene we recall two fundamental
results in this subject --- the Khintchine and Jarn\'ik theorems.

Let us fix a nonincreasing function
$\psi:\mathbb{R}^+\to (0,1]$ which will measure the quality of rational approximations.
A vector $x\in\mathbb{R}^d$ is called {\it $\psi$-approximable} if 
there exist infinitely many $(p,q) \in \Z^d\times \N$ such that
$$
\left\|x - \frac{p}{q}\right\| \le \frac{\psi(q)}{q}.
$$
Let $\mathcal{W}(\R^d,\Q^d,\psi)$ denote the set of $\psi$-approximable vectors in $\mathbb{R}^d$.
Since $\mathcal{W}(\R^d,\Q^d,\psi)$ is the  $\limsup$ set with respect to the balls 
centered at $\frac{p}{q}$ and radius $\frac{\psi(q)}{q}$, a straightforward Borel--Cantelli
argument implies that the set $\mathcal{W}(\R^d,\Q^d,\psi)$ has Lebesgue measure zero provided that 
\begin{equation}\label{eq:cond1}
\sum_{(p,q) \in \Z^d\times \N, \frac{p}{q}\in [0,1)^d} \left(\frac{\psi(q)}{q}\right)^d
=\sum_{q\ge 1} \psi(q)^d<\infty.
\end{equation}
The converse is a fundamental theorem of Khintchine \cite{Khintchine}:

\begin{Theorem}[Khintchine]\label{Khintchine}
If
$$
\sum_{q\ge 1} \psi(q)^d=\infty,
$$
then the set $\mathcal{W}(\R^d,\Q^d,\psi)$ has full Lebesgue measure.
\end{Theorem}

If we strengthen \eqref{eq:cond1} to require that 
$$
\sum_{(p,q) \in \Z^d\times \N, \frac{p}{q}\in [0,1)^d} \left(\frac{\psi(q)}{q}\right)^\alpha=
\sum_{q\ge 1} q^{d-\alpha}\psi(q)^{\alpha}<\infty
$$
for some $\alpha\in (0,d)$,
then it easily follows that the Hausdorff dimension of the set $\mathcal{W}(\R^d,\Q^d,\psi)$
is at most $\alpha$. The converse has been established by Jarn\'ik \cite{Jarnik}:

\begin{Theorem}[Jarn\'ik]\label{Jarnik}
If for some $\alpha\in (0,d)$,
$$
\sum_{q\ge 1} q^{d-\alpha}\psi(q)^{\alpha}=\infty,
$$
then the intersection of $\mathcal{W}(\R^d,\Q^d,\psi)$ 
with every nonempty open subset of $\R^d$ has Hausdorff dimension at least $\alpha$.
\end{Theorem}

\vspace{0.2cm}

The aim of this paper is to develop the metric theory of Diophantine approximation 
on homogeneous varieties of semisimple algebraic groups. In fact, we consider, more generally, 
\begin{itemize}
\item the problem of simultaneous Diophantine approximation with respect to several completions,
\item the problem of Diophantine approximation over number fields by $S$-algebraic integers.
\end{itemize}
Although the question of proving a Khintchine-type theorem for homogeneous varieties
was raised already half a century ago by S. Lang \cite[p.~189]{Lang}, 
it has remained widely open. We are only aware of results for rational quadrics \cite{dr}.
It should be noted that this problem is very
different from problems arising in the theory of Diophantine approximation with dependent quantities, also traditionally called ``Diophantine approximation on manifolds''
(see for instance, \cite{Kleinbock, BD}).
In the latter subject, one studies rational approximation of points on varieties by
{\it all} rational points in $\R^d$, while we are interested in approximation
by rational points lying on the variety.

\vspace{0.2cm}

Let ${\rm X}$ be a nonsingular subvariety of the affine space $A^n$ defined over a number field $K$.
We denote by $V_K$ the set of normalized absolute values $|\cdot|_v$ of $K$ and by $K_v$
the corresponding completions.  We introduce a metric on ${\rm X}(K_v)$:
\begin{equation}\label{eq:dist}
\|x-y\|_v:=\max_{1\le i\le n} |x_i-y_i|_v,
\end{equation}
and the height function on ${\rm X}(K)$:
\begin{equation}\label{eq:height}
\height(x):=\prod_{v\in V_K} \max (1, \|x\|_v).
\end{equation}
For $S \subset V_K$, we set
$$
X_S := \{ (x_v)_{v \in S}~:~x_v \in {\rm X}(K_v);\; \|x_v\|_v\le 1~\text{for almost all}~v\}.
$$
The metric on $X_S$ is defined as the maximum of the local metrics \eqref{eq:dist}.
We fix non-vanishing regular differential form on ${\rm X}$ of top degree. This defines measures $\lambda_v$
on ${\rm X}(K_v)$ and a measure $\lambda_S$ on $X_S$, which is product of the measures on ${\rm X}(K_v)$
normalized so that the subsets $\{\|x_v\|_v\le 1\}$ have measure one.
Different choices of differential forms give equivalent measures.

We are interested in Diophantine approximation in $X_S$ by elements of
${\rm X}(K)$, which are embedded in $X_S$ diagonally.
Let $\Psi=(\psi_v)_{v\in S}$ be a collection of nonincreasing functions $\psi_v:\R^+\to (0,1]$
such that $\psi_v=1$ for almost all $v$. 

\begin{Def}\label{W} We say that a point $x=(x_v)_{v\in S}\in X_S$ is {\it $\Psi$-approximable}
if there are infinitely many $z\in {\rm X}(K)$ such that
\begin{equation}\label{eq:ball}
\|x_v-z\|_v\le \psi_v(\height(z)),\quad v\in S.
\end{equation}
We denote by $\mathcal{W}(X_S, {\rm X}(K),\Psi)$ the set of $\Psi$-approximable points in $X_S$.
\end{Def}
Let
$$
\psi_{S}(t) := \prod_{v \in S}\psi_{v}(t)^{r_v},
$$
with $r_v=2$ if $K_v=\C$ and $r_v=1$ otherwise.
Since ${\rm X}$ is nonsingular, computing volumes with respect to
the measures $\lambda_v$ in local coordinates yields
$$
\lambda_S\left(\{x\in X_S:\, \hbox{~(\ref{eq:ball}) holds}\}\right)\ll_z \psi_{S}(\height(z))^{\dim({\rm X})},
$$
where the implied constant is uniform as $z$ varies in compact sets.
Therefore, the standard Borel--Cantelli argument implies that if 
\begin{equation}\label{eq:el1}
\sum_{z \in {\rm X}(K)\cap D} \psi_{S}(\height(z))^{\dim({\rm X})} < \infty
\end{equation}
for all bounded subsets $D$ of $X_S$, 
then the set $\mathcal{W}(X_S, {\rm X}(K),\Psi)$ has measure zero.

Similarly, one can obtain an elementary estimate on the Hausdorff dimension of the set
$\mathcal{W}(X_S, {\rm X}(K),\Psi)$. Let us assume that $S$ is finite, and  all functions $\psi_v$ are equal.
It follows from the definition of the Hausdorff dimension 
(see Section \ref{sec:h} below) that if with $\alpha>0$,
\begin{equation}\label{eq:el2}
\sum_{z \in {\rm X}(K)\cap D} \psi_v(\height(z))^{\alpha} < \infty
\end{equation}
for all bounded subsets $D$ of $X_S$, 
then the set $\mathcal{W}(X_S, {\rm X}(K),\Psi)$ has Hausdorff dimension at most $\alpha$.

We prove partial converses of these statements in the setting of homogeneous
varieties of simple algebraic groups. Our main results which are optimal in many cases, are illustrated by the following example:

\begin{example}\label{example} 
{\rm
Let 
$$
{\rm X}=\{Q(x)=a\}
$$
be a rational 2-dimensional ellipsoid. We fix a finite prime $p$,
a finite set of primes $p_1,\ldots,p_s$ (possibly including $\infty$)
different from $p$, and $\psi:\R^+\to (0,1]$ a
nonincreasing function. We consider the problem of Diophantine approximation
in ${\rm X}(\Q_{p_1})\times \cdots \times {\rm X}(\Q_{p_s})$ by points in ${\rm X}(\Z[1/p])$.
In this setting, we have:
\begin{enumerate}
\item[(i)] If there exists a bounded subset $D$ of ${\rm X}(\Q_{p_1})\times \cdots \times {\rm
    X}(\Q_{p_s})$ and $\epsilon>0$ such that
$$
\sum_{z \in {\rm X}(\Z[1/p])\cap D} \psi(\height (z))^{2s+\epsilon} = \infty,
$$
then for almost every $(x_1,\ldots,x_s)\in {\rm X}(\Q_{p_1})\times \cdots \times {\rm X}(\Q_{p_s})$,
the system of inequalities
$$
\|x_i-z\|_{p_i}\le\psi(\height(z)),\quad i=1,\ldots, s,
$$
has infinitely many solutions $z\in {\rm X}(\Z[1/p])$.

In view of the above discussion, this result is optimal (up to $\epsilon>0$).
Namely, as we saw in (\ref{eq:el1}), if 
$$
\sum_{z \in {\rm X}(\Z[1/p])\cap D} \psi(\height (z))^{2s} < \infty
$$
for all bounded subset $D$, then the set of such $(x_1,\ldots,x_s)$ has measure zero.

\item[(ii)] If there exists a bounded subset $D$ of ${\rm X}(\Q_{p_1})\times \cdots \times {\rm
    X}(\Q_{p_s})$ and $\alpha\in (0,2s)$ such that
$$
\sum_{z \in {\rm X}(\Z[1/p])\cap D} \psi(\height(z))^{\alpha} = \infty,
$$
then the set of  $(x_1,\ldots,x_s)\in {\rm X}(\Q_{p_1})\times \cdots \times {\rm X}(\Q_{p_s})$
such that the system of inequalities
$$
\|x_i-z\|_{p_i}\le \psi(\height(z)),\quad i=1,\ldots, s,
$$
has infinitely many solutions $z\in {\rm X}(\Z[1/p])$ has Hausdorff dimension at least $\alpha$.

This result is optimal. Indeed, according to (\ref{eq:el2}), if 
$$
\sum_{z \in {\rm X}(\Z[1/p])\cap D} \psi(\height(z))^{\alpha} < \infty
$$
for all bounded subset $D$, then the set of such 
$(x_1,\ldots,x_s)$ has Hausdorff dimension at most $\alpha$.
\end{enumerate}
}\end{example}
 
Our first theorem concerns Diophantine approximation on a
semisimple algebraic group ${\rm G}$. Let $G_{V_K}$ denote the restricted direct product of $G_v$, $v\in V_K$. 
The group of $K$-rational points ${\rm G}(K)$ embeds diagonally in the locally compact group ${G}_{V_K}$ as a discrete
subgroup with finite covolume. A continuous unitary character $\chi$ of ${ G}_{V_K}$ is called automorphic if 
$\chi({\rm G}(K))=1$. Then $\chi$ can be considered as an element of $L^2({ G}_{V_K}/{\rm G}(K))$.
We denote by $L_{00}^2({ G}_{V_K}/{\rm G}(K))$, the subspace of $L^2({ G}_{V_K}/{\rm G}(K))$ orthogonal
to all automorphic characters. The translation action of the group ${ G}_{v}$ on ${ G}_{V_K}/{\rm G}(K)$
defines a unitary representation $\pi_v$ of ${ G}_{v}$ on $L_{00}^2({ G}_{V_K}/{\rm G}(K))$. Fix a
suitable maximal compact subgroup $U_v$ of $G_v$ (as in \cite[\S3]{GGN1}). 
We define the spherical {\it integrability exponent} of $\pi_v$ with respect to $U_v$ by
\begin{equation}\label{eq:q_v}
\mathfrak{q}_v({\rm G}):=\inf\left\{ q>0:\, \begin{tabular}{l}
$\forall$\hbox{ $U_v$-inv. $w\in L_{00}^2({ G}_{V_K}/{\rm G}(K))$}\\
\hbox{$\left< \pi_v(g)w,w\right>\in L^q({\rm G}(K_v))$}
\end{tabular}
\right\}.
\end{equation}
We note that the Langlands programme provides explicit bounds and conjectures
regarding the values of the integrability exponents (see \cite{Sarnak} for a comprehensive discussion).
In particular, it is known that the integrability exponents $\mathfrak{q}_v({\rm G})$ are always finite
(see \cite{Clozel}),
the Ramanujan--Petersson conjecture predicts that $\mathfrak{q}_v({\rm SL}_2)=2$ for all $v$,
and it is known that $\mathfrak{q}_v({\rm SL}_2)=2$ for a positive proportion of $v$'s (see \cite{ram}).
In the setting of Example \ref{example}, it is also known that $\mathfrak{q}_v({\rm G})=2$
due to the celebrated results of Deligne combined with the Jacquet--Langlands correspondence (see \cite[Appendix]{lub}).

For $S\subset V_K$, we set
\begin{equation}\label{eq:sigma_S}
\sigma_S:=\limsup_{N\to \infty} \frac{1}{\log N} |\{v\in S:\, q_v\le N\}|,
\end{equation}
where $q_v$ denotes the cardinality of the residue field for non-Archimedean $v$,
and define
\begin{equation}\label{eq:q_s}
\mathfrak{q}_S({\rm G})=(1+\sigma_S)\sup_{v\in S} \mathfrak{q}_v({\rm G}).
\end{equation}

\begin{Theorem}\label{th:main1}
Let ${\rm G}\subset {\rm GL}_n$ be a connected simply connected almost simple algebraic group 
defined over a number field $K$, $S\subset V_K$, 
and $\Psi=(\psi_v)_{v\in S}$ a collection of nonincreasing functions $\psi_v:\R^+\to (0,1]$
such that $\psi_v=1$ for almost all $v$. Suppose that 
for a bounded subset $D$ of $G_S$ and a constant
$\alpha>\frac{\mathfrak{q}_{V_K \backslash S}({\rm G})}{2}\dim({\rm G})$, we have 
\begin{equation}\label{eq:main1}
\sum_{z \in {\rm G}(K)\cap D} \psi_{S}(\height(z))^{\alpha} = \infty.
\end{equation}
Then the set $\mathcal{W}(G_S, {\rm G}(K),\Psi)$ has full measure in $G_S$.
\end{Theorem}
 
This result is analogous to Khintchine's Theorem (Theorem \ref{Khintchine}).
Comparing \eqref{eq:main1} with \eqref{eq:el1}, we conclude that it is optimal when  
$\mathfrak{q}_{V_K \backslash S}({\rm G})=2$. The assumption that ${\rm G}$ is simply connected
is essential here, and Theorem \eqref{th:main1} can be considered as a quantitative version
of the strong approximation property \cite[7.4]{PlaRa},
which is only valid for simply connected semisimple groups.

\vspace{0.2cm}

More generally, we consider a quasi-affine variety ${\rm X}\subset A^n$ defined over a number field $K$ 
and equipped with a transitive action of a connected almost simple 
algebraic group ${\rm G}\subset\hbox{GL}_n$ defined over $K$.
For $S\subset V_K$ and a collection of functions $\Psi=(\psi_v)_{v\in S}$, 
we are interested in analyzing the set $\mathcal{W}(X_S, {\rm X}(K),\Psi)$ defined as above.
Even when the set ${\rm X}(K)$ is not discrete in $X_S$, 
it might happen that its closure has measure zero in $X_S$, 
and in particular, the direct analogue of Theorem \ref{th:main1} fails. To describe the structure of
$\mathcal{W}(X_S, {\rm X}(K),\Psi)$, we observe that $X_S$ is a union of open sets
$$
X_{S,S'}:= \{x\in X_S: \|x_v\|_v\le 1~\hbox{for all $v\in S\backslash S'$}\},
$$
where $S'$ runs over finite subsets of $S$ containing the Archimedean places of $S$.
Hence, the problem reduces to analyzing the sets
$$
\mathcal{W}(X_{S,S'}, X_{S,S'}\cap {\rm X}(K),\Psi)\subset X_{S,S'}.
$$
When $S'$ contains all
places $v$ such that $\psi_v\ne 1$, it is easy to check that
\begin{align*}
\mathcal{W}(X_{S,S'}, X_{S,S'}\cap {\rm X}(K),(\psi_v)_{v\in S})
=&\mathcal{W}(X_{S'}, X_{S,S'}\cap {\rm X}(K),(\psi_v)_{v\in S'})\\
&\times\left (\prod_{v\in S\backslash S'} \{x\in {\rm
    X}(K_v):~\|x\|_v\le  1\}\right).
\end{align*}
As we will see below (cf. Lemma \ref{l:discrete}), if $X_{S,S'}\cap {\rm X}(K)$ is not discrete in $X_{S'}$,
then the closure of $X_{S,S'}\cap {\rm X}(K)$, embedded in $X_{S'}$, is open in $X_{S'}$,
and we shall prove an analogue of Theorem \ref{th:main1} for the sets
$$
\mathcal{W}(X_{S'}, X_{S,S'}\cap {\rm X}(K),(\psi_v)_{v\in S'})\subset X_{S'}.
$$
To state this result, we introduce a measure of the growth of the number of rational points in $X_{S,S'}$
defined by
\begin{align}\label{eq:ass}
\mathfrak{a}_{S,S'}({\rm X}) &:= \sup_{D} \limsup_{h\to \infty} \frac{\log |\{z\in {\rm X}(K)\cap D:\, \height(z)\le h \}|)}{\log h},
\end{align}
where $D$ runs over bounded subset of $X_{S,S'}$.
We note that $\mathfrak{a}_{S,S'}({\rm X})<\infty$.
When the supremum in (\ref{eq:ass})
is taken over all bounded subsets of $X_S$ we use the notation $\mathfrak{a}_{S}({\rm X})$. In the case of a group variety, 
if ${\rm G}$ is isotropic over $V_K\backslash S$ then 
$\mathfrak{a}_{S}({\rm G})>0$.

\begin{Theorem}\label{th:main2}
Let  ${\rm X}\subset A^n$ be a quasi-affine variety defined over a number field $K$ 
and equipped with a transitive action of a connected almost simple 
algebraic group ${\rm G}\subset\hbox{\rm GL}_n$ defined over $K$.
Let $S\subset V_K$, $S'$ be a finite subset of $S$
containing the Archimedean places of $S$, and let $\Psi=(\psi_v)_{v\in S'}$ be
a collection of nonincreasing functions $\psi_v:\R^+\to (0,1]$.
Suppose that 
for a bounded subset $D$ of $X_{S,S'}$ and a constant
\begin{equation}\label{eq:alpha}
\alpha > \frac{\mathfrak{a}_{S, S'}({\rm X})}{\mathfrak{a}_{S}({\rm G})}\frac{\mathfrak{q}_{V_K \backslash S}({\rm
    G})}{2}\dim({\rm X}),
\end{equation}
we have
\begin{equation}\label{eq:main4}
\sum_{z \in {\rm X}(K)\cap D}
\psi_{S'}(\height(z))^{\alpha} = \infty.
\end{equation}
Then $\overline{X_{S,S'}\cap {\rm X}(K)}$ is open in $X_{S'}$, and 
the set  $\mathcal{W}(X_{S'}, X_{S,S'}\cap {\rm X}(K),\Psi)$ has full measure in 
$\overline{X_{S,S'}\cap {\rm X}(K)}$.
\end{Theorem}

Estimate \eqref{eq:alpha}
contains an interesting interplay between the arithmetic datum 
$\frac{\mathfrak{a}_{S, S'}({\rm X})}{\mathfrak{a}_{S}({\rm G})}$,
which measures the growth rates of rational points, and the analytic datum
$\mathfrak{q}_{V_K \backslash S}({\rm G})$, which measures 
the spectral gap for automorphic representations.
Typically,  $\frac{\mathfrak{a}_{S, S'}({\rm X})}{\mathfrak{a}_{S}({\rm G})}\le 1$ and 
$\frac{\mathfrak{q}_{V_K \backslash S}({\rm G})}{2}\ge 1$, but their product
must be always at least one. Indeed, if this is not the case, then
$$
\frac{\mathfrak{a}_{S, S'}({\rm X})}{\mathfrak{a}_{S}({\rm G})}\frac{\mathfrak{q}_{V_K \backslash S}({\rm
    G})}{2}\dim({\rm X})<\dim({\rm X}),
$$
and one can exhibit a family of approximation functions $\psi_v$ such that both \eqref{eq:el1} and
\eqref{eq:main4} hold. Since this contradicts Theorem \ref{th:main2}, we conclude that
$$
\mathfrak{q}_{V_K \backslash S}({\rm G})\ge 2\frac{\mathfrak{a}_{S, S'}({\rm X})}{\mathfrak{a}_{S}({\rm G})},
$$
which was also observed in \cite[Cor.~1.8]{GGN1}. If the equality holds (see, for instance, 
Example \ref{example}), then the result of Theorem \ref{th:main2} is optimal.

\vspace{0.2cm}

Now we state an analogue of the Jarn\'ik theorem (Theorem \ref{Jarnik})
which estimates the Hausdorff dimension of the sets $\mathcal{W}(X_{S'}, X_{S,S'}\cap {\rm X}(K),\Psi)$.
The problem of estimating the Hausdorff dimension in a nonconformal setting
is quite subtle, and in particular, there is no version of Jarn\'ik's theorem
for general rectangular regions defined by a family of functions $\psi_1,\ldots,\psi_d$.
Therefore, we restrict our attention to the case when all the functions $\psi_v$ are equal
to a single function $\psi$ and write $\mathcal{W}(X_{S'}, X_{S,S'}\cap {\rm X}(K),\psi)$
to denote the set of approximable points.

\begin{Theorem}\label{th:main3}
Let  ${\rm X}\subset A^n$ be a quasi-affine variety defined over a number field $K$ 
and equipped with a transitive action of a connected almost simple 
algebraic group ${\rm G}\subset\hbox{\rm GL}_n$ defined over $K$.
Let $S\subset V_K$, $S'$ be a finite subset of $S$
containing the Archimedean places of $S$, and let $\psi:\R^+\to (0,1]$ be a nonincreasing function.
Suppose that 
for some $0<\alpha<\sum_{v\in S'} r_v\dim({\rm X})$ and a bounded subset $D$ of $X_{S,S'}$,
we have
\begin{equation}\label{eq:hh}
\sum_{z \in {\rm X}(K)\cap D}
\psi(\height(z))^{\alpha} = \infty.
\end{equation}
Then $\overline{X_{S,S'}\cap {\rm X}(K)}$ is open in $X_{S'}$,
and the intersection of the set $\mathcal{W}(X_{S'}, X_{S,S'}\cap {\rm X}(K),\Psi)$
with every nonempty open subset of $\overline{X_{S,S'}\cap {\rm X}(K)}$
has Hausdorff dimension at least
$$
\frac{2\mathfrak{a}_{S}({\rm G})}{\mathfrak{a}_{S, S'}({\rm X})\mathfrak{q}_{V_K \backslash
    S}({\rm G})} \cdot \alpha.
$$
\end{Theorem}

We note that this estimate is optimal if 
$\mathfrak{q}_{V_K \backslash S}({\rm G})= 2\frac{\mathfrak{a}_{S, S'}({\rm X})}{\mathfrak{a}_{S}({\rm G})}$,
which holds for a number of cases (see, for instance, Example \ref{example}).

\subsection*{Acknowledments}
The first author acknowledges support of  EPSRC. The second author acknowledges
  support of EPSRC, ERC and RCUK. The third author acknowledges support of ISF.

\section{Notation}
Let $K$ be a number field.
We denote by $V_K$ the set of normalised absolute values of $K$
and by $K_v$ the corresponding completions. 
The subsets of Archimedean and non-Archimedean absolute values are denoted by 
$V_K^\infty$ and $V_K^f$ respectively.
For non-Archimedean $v$, we denote by $O_v$
the ring of integers in $K_v$. 
For a subset $T\subset V_K$, we introduce the ring $O_T$ of $T$-integers:
$$
O_T := \{x \in K~:~|x|_v\le 1~\text{for non-Archimedean}~v \notin T\}.
$$

Let ${\rm G}\subset GL_n$ be a connected almost simple linear algebraic group
defined over $K$.
For $v\in V_K$, we set $G_v={\rm G}(K_v)$, and for $S\subset V_K$,
$G_S$ denotes the restricted direct product of $G_v$, $v\in S$ with respect to ${\rm G}(O_v)$.

When $S=S_1\sqcup S_2$, we have $G_{S}=G_{S_1}\times G_{S_2}$,
and in order to simplify notation we often identify a subset $B$ of $G_{S_1}$
with the subset $B\times \{e\}$ of $G_{S}$.

We define a height function $\height$ on $G_{V_K}$ by
\begin{equation}\label{eq:hhh}
\height(g) :=\prod_{v\in V_K} \max (1,\|g_v\|_v), \quad g\in G_{V_K}.
\end{equation}
This extends the definition of the height function from (\ref{eq:height}). 

For each $v$, we fix a good special maximal compact subgroup $U_v \subset G_v$
so that $U_v = {\rm G}(O_v)$ for all but finitely many $v$.
For $S\subset V^f_K$, we set $U_S=\prod_{v\in S} U_v$.
Each group $G_v$ is equipped with an invariant measure $m_v$ which we normalize for
non-Archimedean $v$ so that $m_v(U_v)=1$. The groups $G_S$ are equipped with the corresponding
product measures $m_S$. In particular, we denote $m_{V_K}$ by $m$.
Since ${\rm G}$ is semisimple, the subgroup ${\rm G}(K)$ has finite covolume in $G_{V_K}$.
We denote by $\mu$ the invariant probability measure on the quotient
space $\Upsilon:=G_{V_K}/{\rm G}(K)$.

\section{Group varieties}\label{sec:group}

Our proof of the main theorem is based on analysis of suitable averaging operators
on the space $L^2(\Upsilon)$ of square-integrable functions of $\Upsilon$.
For $T\subset V_K$ and a probability measure $\beta$ on $G_{T}$, we introduce
the averaging operator $\pi_{T}(\beta):L^2(\Upsilon)\to L^2(\Upsilon)$ defined by
\begin{equation}
\pi_{T}(\beta)\phi(\varsigma)=\int_{G_{T}}\phi(g^{-1}\varsigma)\,d\beta(g), \quad \phi\in L^2(\Upsilon).
\end{equation}
In the case when ${\rm G}$ is simply connected, the asymptotic behaviour of these
averaging operators is described by

\begin{Theorem}[\cite{GGN1}, Th. 4.2]\label{cor:mean_simply connected}
Assume that ${\rm G}$ is simply connected, and let $\beta$ be the Haar-uniform probability
measure supported on bi-$U_{V_K\backslash S}$-invariant bounded subset $B$ of $G_{V_K\backslash S}$.
Then for every $\phi\in L^2(\Upsilon)$,
$$
\left\| \pi_{V_K\backslash S}(\beta)\phi-\int_\Upsilon\phi\,d\mu\right\|_2\ll_{\delta}
m_{V_K\backslash S}(B)^{-\frac{1}{\mathfrak{q}_{V_K\backslash S}({\rm G})}+\delta}\|\phi\|_2
$$
for every $\delta>0$.
\end{Theorem}

The connection between behavior of the averaging operators and Diophantine approximation
is provided by the following proposition.
We use the notation $I_v=\{q_v^n\}_{n>0}$ for $v\in V_K^f$, where $q_v$ denotes the cardinality
of the residue field, and $I_v=(0,1)$ for $v\in V_K^\infty$. 

\begin{Prop}[\cite{GGN1}, Prop.~5.3] \label{p:dual}
Fix $S\subset V_K$, finite $S'\subset S$ and a bounded subset $\Omega$ of $G_S$.
Then there exists a family of measurable subsets
$\Phi_\epsilon$ of $\Upsilon$ indexed by $\epsilon=(\epsilon_v)_{v\in S'}$,
where $\epsilon_v\in I_v$,
that satisfies
\begin{equation}\label{eq:llow}
 \prod_{v\in S'}  \epsilon_v^{r_v\dim ({\rm G})}\ll \mu(\Phi_\epsilon)\ll \prod_{v\in S'}  \epsilon_v^{r_v\dim ({\rm G})}
\end{equation}
and the following property holds:

if for a subset $B\subset G_{V_K\backslash S}$,
$\epsilon=(\epsilon_v)_{v\in S'}$ as above, $x\in \Omega$ and $\varsigma:=(e,x^{-1}){\rm G}(K)\in \Upsilon$,
we have 
\begin{equation*}
B^{-1}\varsigma\cap \Phi_\epsilon\ne \emptyset,
\end{equation*}
then there exists $z\in {\rm G}(K)$ such that
\begin{equation}\label{eq:cl1}
\height(z)\le c_0 \sup_{b\in B} \height(b)
\end{equation}
and
\begin{align*}
&\|x_v-z\|_v\le \epsilon_v \quad\quad\hbox{for all $v\in S'$,}\\
&\|x_v-z\|_v\le 1 \quad\quad\hbox{for all $v\in S\backslash S'$.}\nonumber
\end{align*}
\end{Prop}

We note that the upper bound in \eqref{eq:llow} was not stated explicitly in \cite{GGN1},
but it follows easily from the construction of the sets $\Phi_\epsilon$.

\vspace{0.2cm}

\subsection*{Proof of Theorem \ref{th:main1}}
We note that the divergence assumption \eqref{eq:main1} implies that the group ${\rm G}$ is
isotropic over $V_K\backslash S$. Indeed, if ${\rm G}$ is anisotropic over $V_K\backslash S$,
then $G_{V_K\backslash S}$ is compact (see \cite[\S3.1]{PlaRa}).
Since ${\rm G}(K)$ embeds discretely in $G_{V_K}$, the number of
${\rm G}(K)$-points in $G_{V_K\backslash S}\times D$ is finite,
and this contradicts \eqref{eq:main1}.

Since ${\rm G}$ is isotropic over $V_K\backslash S$, 
it follows from the strong approximation property of 
simply connected groups (see \cite[\S7.4]{PlaRa}) that ${\rm G}(K)$ is dense in $G_S$.
Hence, the claim of the theorem follows immediately 
if $\psi_v(t)\nrightarrow 0$ as $t\to \infty$ for all $v\in S$.
From now on, we assume that $\psi_v(t)\rightarrow 0$ as $t\to \infty$ for at least one $v\in S$.

For $v\in V_K^f$, 
given a nonincreasing function $\psi_v:\R^+\to (0,1]$, one can construct a nonincreasing function
$\tilde \psi_v:\R^+\to (0,1)$ such that $\tilde \psi_v\le \psi_v\le q_v \tilde\psi_v$  and
$\hbox{Im}(\tilde \psi_v)\subset I_v$.
Now if we replace $\psi_v\ne 1$ in $\Psi$ by $\tilde \psi_v$, then the 
divergence assumption (\ref{eq:main1}) still holds. 
Since $\tilde \psi_v\le \psi_v$, the theorem for $\tilde \psi_v$'s implies 
the theorem of $\psi_v$'s. Hence, we may assume without loss of 
generality that $\hbox{Im}(\psi_v)\subset I_v$ for all $\psi_v\ne 1$.

To simply notation, we set 
$$
d:=\dim({\rm G}),\;\; \mathfrak{q}:=\mathfrak{q}_{V_K \backslash S}({\rm G}),\;\; \mathfrak{a}:=\mathfrak{a}_{S}({\rm G}).
$$


We consider a family of bounded bi-$U_{V_K \backslash S}$-invariant subsets of $G_{V_K\backslash S}$
defined by
$$
B_h := U_{V_K \backslash S}\{ g \in G_{V_K \backslash S}~:~\height(g) \le  h \}U_{V_K  \backslash S}.
$$
Let 
$$
U:=\{g\in G_{V_K\backslash S}:\, \|g_v-e\|\le 1\quad\hbox{for $v\in {V_K\backslash S}$}\}.
$$
Since $U$ and $U_{V_K \backslash S}$ are compact, there exists $c_1\ge 1$ such that
\begin{equation}\label{eq:h2}
\sup_{b\in U^{-1}B_{h}} \height(b)\le c_1\, h.
\end{equation}
It follows from the definition of $\mathfrak{a}=\mathfrak{a}_{S}({\rm G})$ 
that for every $\delta>0$ and sufficiently large $h$,
\begin{equation}\label{eq:rational}
|{\rm G}(K)\cap D\cap \{h/2<\height\le h\}|\le h^{\mathfrak{a}+\delta}.
\end{equation}
Since the function $\psi_S$ is nonincreasing,
it follows from (\ref{eq:main1}) that
$$
\sum_{n = 1}^{\infty} |{\rm G}(K)\cap D\cap \{2^{n-1}<\height\le 2^n\}|\psi_S(2^{n-1})^{\alpha} = \infty,
$$
and because of \eqref{eq:rational} we also get
\begin{equation}\label{newsum1}
\sum_{n = 1}^{\infty} 2^{(\mathfrak{a}+\delta)n}\psi_S(2^{n-1})^{\alpha} = \infty.
\end{equation}
Since $0< \psi_S\le 1$, there exists $\alpha_0(\delta)\in [0,\infty]$ such that this series converges
for all $\alpha>\alpha_0(\delta)$ and diverges for all $\alpha<\alpha_0(\delta)$.
We fix $\alpha_0>\mathfrak{q}d/2$ such that series (\ref{eq:main1}) diverges for $\alpha=\alpha_0$.
Since divergence in \eqref{eq:main1} implies divergence in \eqref{newsum1},
we have $\alpha_0(\delta)\ge \alpha_0$.

Using that the function $\psi_S$ is monotone, 
it is easy to check that \eqref{newsum1} is equivalent to 
\begin{equation}\label{newsum11}
\sum_{n = 1}^{\infty} 2^{(\mathfrak{a}+\delta)n}\psi_S(c\,2^{n-1})^{\alpha} = \infty.
\end{equation}
with any $c>0$. We choose $c=c_0 c_1/2$ where $c_0$ is as in (\ref{eq:cl1}), and
$c_1$ is as in \eqref{eq:h2}.

In order to continue our argument, it would be convenient to know that $\alpha_0(\delta)<\infty$.
This is arranged by the following construction. 
If 
$$
2^{(\mathfrak{a}+\delta)n}\psi_S(c\,2^{n-1})^{\alpha_0}\to 0\quad\hbox{as $n\to \infty$},
$$
we set $R_n:=\psi_S(c\,2^{n-1})$. Otherwise, there exists $\epsilon>0$ such that the set
$$
N:=\{n:\, 2^{(\mathfrak{a}+\delta)n}\psi_S(c\,2^{n-1})^{\alpha_0}\ge \epsilon\}
$$
is infinite. In this case, we set 
$$
R_n:=
\left\{
\begin{tabular}{ll}
$(\epsilon\, 2^{-(\mathfrak{a}+\delta)n})^{1/\alpha_0}$ & \hbox{for $n\in N$,}\\
$\psi_S(c\,2^{n-1})$ & \hbox{otherwise.}
\end{tabular}
\right.
$$
In both cases, we have $R_n\le \psi_S(c\,2^{n-1})$, and since $N$ is infinite,
\begin{equation}\label{newsum111}
\sum_{n = 1}^{\infty} 2^{(\mathfrak{a}+\delta)n}R_n^{\alpha} = \infty
\end{equation}
for $\alpha=\alpha_0$.
There exists $\alpha_1(\delta)\in [0,\infty]$ such that this series converges
for all $\alpha>\alpha_1(\delta)$ and diverges for all $\alpha<\alpha_1(\delta)$.
Clearly, $\alpha_1(\delta)\ge \alpha_0$.
It also follows from the definition of $R_n$ that 
$R_n\ll 2^{-\kappa n}$ with some $\kappa>0$. Hence, $\alpha_1(\delta)<\infty$.

Let $S'$ be the  finite subset of $S$ on which $\psi_v\ne 1$.
Recall that 
$$
\psi_S(t)=\prod_{v\in S'} \psi_v(t)^{r_v}.
$$
Since $R_n\le \psi_S(c\,2^{n-1})$, we may write 
$$
R_n=\prod_{v\in S'} (\epsilon_v^{(n)})^{r_v},
$$
where $\epsilon_v^{(n)}\le \psi_v(c\,2^{n-1})$ for all $v\in S'$.

Let $\alpha<\alpha_1(\delta)$.
We denote by $\Phi_n$ the collection of measurable subsets of $\Upsilon$
defined by Proposition \ref{p:dual} with $\epsilon_v=\epsilon_v^{(n)}$, $v\in S'$.
We consider the sequence of functions on $\Upsilon$ defined by
$$
\phi_n := c_n 1_{\Phi_n}\quad\hbox{with}\quad c_n=2^{(\mathfrak{a}+\delta)n}
R_n^{\alpha-d}.
$$
Since according to (\ref{eq:llow}), we have $\mu(\Phi_n)\gg R_n^{d}$
and $\alpha<\alpha_1(\delta)$, it follows that
\begin{equation}\label{newsum2}
\sum_{n\ge 1}\int_{\Upsilon} \phi_n\,d\mu=\sum_{n\ge 1} c_n\mu(\Phi_n)=\infty.
\end{equation}

Let  $\beta_n$ denote the Haar-uniform probability measure supported on $B_{2^n}$
and 
$$
F_{k} := \sum_{n \geq k} \left|\pi_{V_K\backslash S}(\beta_n)\phi_{n} - \int_{\Upsilon} \phi_n\,d\mu \right|.
$$
We claim that $F_k$ is $L^2$-integrable for sufficiently small $\delta>0$ and $\alpha<\alpha_1(\delta)$
sufficiently close to $\alpha_1(\delta)$.
To verify this, we observe that by Theorem \ref{cor:mean_simply connected} and (\ref{eq:llow}), for every $\delta>0$,
\begin{align*}
\| F_k \|_2 
&\ll_{\delta} \sum_{n \geq k} m_{V_K\backslash S}(B_{2^n})^{-1/\mathfrak{q}+\delta}\|\phi_n\|_{2}\\
&= \sum_{n \geq k} m_{V_K\backslash S}(B_{2^n})^{-1/\mathfrak{q}+\delta}c_n\mu(\Phi_n)^{1/2}\\
&\le \sum_{n \geq k} m_{V_K\backslash S}(B_{2^n})^{-1/\mathfrak{q}+\delta}2^{(\mathfrak{a}+\delta)n}
R_n^{\alpha-d/2}.
\end{align*}
By \cite[Lem.~6.1]{GGN1}, for every $\delta > 0$ and
sufficiently large $n$,
$$
m_{V_K \backslash S}(B_{2^n}) \gg 2^{(\mathfrak{a} - \delta)n}.
$$
Hence, for sufficiently large $k$,
\begin{align*}
\| F_k \|_2 
&\ll_\delta \sum_{n \geq k} 2^{(a(1-1/\mathfrak{q})+\theta)n} R_n^{\alpha-d/2}\\
&= \sum_{n \geq k} 2^{(\mathfrak{a}(1-1/\mathfrak{q})+2\theta)n} R_n^{\alpha-d/2}\cdot 2^{-\theta n},
\end{align*}
where $\theta=\theta(\delta)$ satisfies $\theta(\delta)\to 0^+$ as $\delta\to 0^+$.
We apply H\"older's inequality to the above sum with the exponents
$$
r=\mathfrak{a}/(\mathfrak{a}(1-1/\mathfrak{q})+2\theta)\quad\hbox{and}\quad \bar r=(1-1/r)^{-1}.
$$
Note that when $\delta$ is sufficiently small, we have $r>1$. Then we obtain
\begin{align}\label{eq:last}
\| F_k \|_2 
&\ll_\delta
\left(\sum_{n \geq k} 2^{\mathfrak{a} n} R_n^{r(\alpha-d/2)}\right)^{1/r}
\cdot \left( \sum_{n \geq k} 2^{-\theta \bar r n}\right)^{1/\bar r}.
\end{align}
Since $\alpha_0>\mathfrak{q}d/2$, 
$$
\frac{\alpha_0-d/2}{1-1/\mathfrak{q}}>\alpha_0.
$$
This implies that 
$$
\frac{\mathfrak{a}(\alpha-d/2)}{\mathfrak{a}(1-1/\mathfrak{q})+2\theta}>\alpha
$$
for every sufficiently small $\delta>0$ and every $\alpha\ge \alpha_0$.
Since $\alpha_1(\delta)\ge \alpha_0$, it follows that for $\alpha<\alpha_1(\delta)$
sufficiently close to $\alpha_1(\delta)$ and
for sufficiently small $\delta>0$, we also have 
$$
r(\alpha-d/2)=\frac{\mathfrak{a}(\alpha-d/2)}{\mathfrak{a}(1-1/\mathfrak{q})+2\theta}>\alpha_1(\delta).
$$
This implies that the series in \eqref{eq:last} converges, and we 
conclude that $F_k$ is $L^2$-integrable. 

Let $\delta>0$ and  $\alpha<\alpha_1(\delta)$ be such that $F_k$ is $L^2$-integrable.
We consider a sequence of subsets 
$$ \Upsilon_n := \{ \varsigma \in \Upsilon~:~ B^{-1}_{2^n} \varsigma \cap \Phi_n = \emptyset\}$$
of $\Upsilon$. Note that by \eqref{newsum2}, on the set $\bigcap_{n \geq k} \Upsilon_n$,
$$
F_{k} = \sum_{n \geq k} \int_{\Upsilon} \phi_n\,d\mu=\infty.
$$
Since $F_k$ is $L^2$-integrable, it follows that
$\mu\left(\bigcap_{n \geq k} \Upsilon_n\right) = 0$ and the set
$\Upsilon_\infty := \liminf (\Upsilon_n)$ also has measure zero.
We denote by  $\widetilde{\Upsilon}_{\infty}$ the pre-image of $\Upsilon_\infty$ in $G_{V_K}$.
Then $m(\widetilde{\Upsilon}_\infty) = 0$. 

Let $\Omega$ be a compact subset of $G_S$, and
$$
\Omega'=\{x\in \Omega:\, \exists~ y\in U:\, (y,x^{-1}) \notin
\widetilde{\Upsilon}_\infty \}.
$$
Since
$$
(U\times (\Omega\backslash \Omega')^{-1})\subset \widetilde{\Upsilon}_\infty,
$$
and $U$ has positive measure,
it follows that the set $\Omega\backslash \Omega'$ has measure zero.
Let us take an exhaustion $G_S=\cup_{j\ge 1} \Omega_j$ of $G_{S}$ by compact sets.
Then $\cup_{j\ge 1} \Omega_j'$ is a subset of $G_{S}$ of full measure.
Hence, it suffices to show that given a compact subset $\Omega$ of $G_S$,
almost every element of $\Omega'$ is contained in $\mathcal{W}(G_S,{\rm G}(K),\Psi)$. 

For $x\in \Omega'$, there exists $y\in U$ such that
$\tilde \varsigma:=(y,x^{-1})\notin \widetilde{\Upsilon}_\infty$.
Then
$$
\tilde \varsigma{\rm G}(K)\notin \Upsilon_\infty = \liminf (\Upsilon_n),
$$
and $\tilde \varsigma{\rm G}(K)\notin \Upsilon_n$ for infinitely many $n$.
This means that for infinitely many $n$, we have 
$$
B_{2^n}^{-1}\tilde \varsigma{\rm G}(K) \cap \Phi_n \ne \emptyset,
$$
and
$$
(U^{-1}B_{2^n})^{-1}(e,x){\rm G}(K) \cap \Phi_n \ne \emptyset.
$$
Now we are in position to apply Proposition \ref{p:dual}. It follows that 
for infinitely many $n$, there exists $z_n\in {\rm G}(K)$ such that
$$
\height(z_n)\le c_0 \sup_{b\in U^{-1}B_{2^n}} \height(b)\le c_0c_1\, 2^n=c\, 2^{n-1},
$$
and
\begin{align*}
&\|x_v-z_n\|_v\le \epsilon_v^{(n)}\le \psi_v(c\,2^{n-1}) \le \psi_v(\height(z_n)) \quad\hbox{for all $v\in S'$,}\\
&\|x_v-z_n\|_v\le 1\quad\hbox{for all $v\in S\backslash S'$.}
\end{align*}
Recall that $\psi_v(t)\to 0$ as $t\to \infty$ for at least one $v$.
Since then
$$
\|x_v-z_n\|_v\le  \psi_v(c\,2^{n-1})\to 0,
$$
we conclude that if $x_v\notin {\rm G}(K)$, then the set $\{z_n\}$ must be infinite.
This proves that almost every element in $\Omega'$ belongs to $\mathcal{W}(G_S,{\rm G}(K),\Psi)$,
and finishes the proof of Theorem \ref{th:main1}.

\section{Homogeneous Varieties } \label{sec:variety}

In this section, ${\rm X}\subset A^n$ is a quasi-affine algebraic variety defined over a number field
$K$ and equipped with 
a transitive action of a connected almost simple algebraic group ${\rm G}\subset\hbox{GL}_n$
defined over $K$. Because ${\rm G}$ is not assumed to be simply connected,
it usually has nontrivial automorphic characters, and the behaviour of
the averaging operators is more subtle than in Theorem \ref{cor:mean_simply connected}. 

Let $\mathcal{X}_{aut}(G_{V_K})$ be the set of automorphic characters of $G_{V_K}$,
namely the set consisting of continuous unitary characters $\chi$ of $G_{V_K}$ such that
$\chi({\rm G}(K))=1$.  Let $S'$ be a finite subset of $V_K$.
For an open subgroup $U$ of $G_{V^f_K\backslash S'}$, we denote by $\mathcal{X}_{aut}(G_{V_K})^U$ the subset of
$\mathcal{X}_{aut}(G_{V_K})$ consisting of $U$-invariant characters. 
By \cite[Lem.~4.4]{GGN1}, the set $\mathcal{X}_{aut}(G_{V_K})^U$ is finite.
We denote by $G^U$ the kernel of $\mathcal{X}_{aut}(G_{V_K})^U$ in $G_{V_K}$. 
Then $G^U$ is a finite index subgroup in $G_{V_K}$ (see \cite[Lem.~4.4]{GGN1}), which clearly contains $G(K)$. 

The following theorem describes the asymptotic behavior of the averaging operators acting in $L^2(\Upsilon)$
for ${\rm G}$ which is not necessarily simply connected.

\begin{Theorem}[\cite{GGN1}, Th.~4.5]\label{c:mean}
Let $S$ be a subset of $V_K$ and  $S'$ a finite subset of $S$.
Let $U^0$ be a finite index subgroup of $U_{V_K^f\cap (S\backslash S')}$
and $U=U_{V_K^f\backslash S}U^0$.
Let $B$ be a bounded measurable subset of $G_{V_K\backslash S}\cap G^U$
which is $U_{V_K^f\backslash S}$-biinvariant and
$\beta$ the Haar-uniform probability measure supported on 
the subset $U^0B$ of $G_{(V_K\backslash S)\cup (V^f_K\backslash S')}$.
Then for every $\phi\in L^2(\Upsilon)$ such that 
$\supp(\phi)\subset G^U/{\rm G}(K)$, we have
$$
\left\|\pi_{(V_K\backslash S)\cup (V^f_K\backslash S')}(\beta)\phi -\left(\int_\Upsilon \phi\,
    d\mu\right)\xi_U \right\|_2\ll_{\delta}
m_{V_K\backslash S}(B)^{-\frac{1}{\mathfrak{q}_{V_K\backslash S}({\rm G})}+\delta}\|\phi\|_2
$$
for every $\delta>0$, where $\xi_U$ is the function on $\Upsilon$
such that $\xi_U=|G_{V_K}:G^{U}|$ on the open set $G^{U}/{\rm G}(K)\subset \Upsilon$
and $\xi_U=0$ otherwise.
\end{Theorem}

The version of Proposition \ref{p:dual} for general homogeneous varieties is provided by

\begin{Prop}[\cite{GGN1}, Prop.~5.5]\label{p:dual2}
Fix $S\subset V_K$, finite $S'\subset S$, $x^0\in X_{S'}$, and a bounded subset $\Omega$ of $G_{S'}$.
Then there exist $\epsilon_0\in (0,1)$ and
 a family of measurable subset $\Phi_\epsilon$ of $\Upsilon$ indexed by
$\epsilon=(\epsilon_v)_{v\in S'}$ where $\epsilon_v\in I_v\cap (0,\epsilon_0)$ that satisfy
\begin{equation}\label{eq:low1}
\prod_{v\in S'} \epsilon_v^{r_v \dim ({\rm X})}\ll\mu(\Phi_\epsilon)\ll \prod_{v\in S'} \epsilon_v^{r_v \dim ({\rm X})}
\end{equation}
and the following property holds:

if for $B\subset G_{V_K\backslash S}\times \prod_{v\in V_K^f\cap (S\backslash S')} {\rm G}(O_v)$,
$\epsilon=(\epsilon_v)_{v\in S'}$ as above, and
$\varsigma:=(e,g^{-1}){\rm G}(K)\in \Upsilon$ with $g\in \Omega$,  we have
\begin{equation*}
B^{-1}\varsigma\cap \Phi_\epsilon\ne \emptyset,
\end{equation*}
then there exists $\gamma\in {\rm G}(O_{(V_K\backslash S)\cup S'})$ such that
\begin{equation}\label{eq:low2}
\height(\gamma)\le c_0\, \sup_{b\in B} \height(b)
\end{equation}
and for $x=gx^0\in X_{S'}$
\begin{equation*}
\|x_v-\gamma x^0_v\|\le \epsilon_v\quad\hbox{for all $v\in S'$.}
\end{equation*}
\end{Prop}
The upper bound in \eqref{eq:low1} was not stated in \cite{GGN1}, but it follows from 
the explicit construction of the sets $\Phi_\epsilon$.

\begin{Lemma}\label{l:discrete}
Let $S$ be a subset of $V_K$ and $S'$ a finite subset of $S$ containing the Archimedean places of $S$.
Suppose that the set $X_{S,S'}\cap {\rm X}(K)$, embedded in $X_{S'}$, is not discrete.
Then the group ${\rm G}$ is isotropic over $V_K\backslash S$, and 
$\overline{X_{S,S'}\cap {\rm X}(K)}$ is open in $X_{S'}$.
\end{Lemma}

\begin{proof}
We first show that $X_{V_K\backslash S}$ must be noncompact. Indeed, 
since $X_{S,S'}\cap {\rm X}(K)$ is not discrete in $X_{S'}$, there exists
a bounded subset $D$ of $X_{S,S'}$ which contains infinitely many elements of 
$X_{S,S'}\cap {\rm X}(K)$. On the other hand, ${\rm X}(K)$ is discrete in $X_{V_K}$,
and if $X_{V_K\backslash S}$ were compact,
then there would have been only finitely many elements of ${\rm X}(K)$ contained in
$X_{V_K\backslash S}\times D$, which gives a contradiction.

If ${\rm G}$ is anisotropic over 
$V_K\backslash S$, then $V_K\backslash S$ is finite by \cite[Th.~6.7]{PlaRa},
and $G_{V_K\backslash S}$ is compact by \cite[Th.~3.1]{PlaRa}.
It follows from finiteness of Galois cohomology over local fields \cite[Th.~6.14]{PlaRa}
that $X_{V_K\backslash S}$ consists of finitely many orbits of $G_{V_K\backslash S}$. 
Then $X_{V_K\backslash S}$ is compact contradicting the previous paragraph.
This shows that ${\rm G}$ must be isotropic over $V_K\backslash S$.

To describe the structure of $\overline{X_{S,S'}\cap {\rm X}(K)}$ in $X_{S'}$,
we note that 
$$
X_{S,S'}\cap {\rm X}(K)= {\rm X}(O_{(V_K\backslash S)\cup S'} ).
$$
Let $p: \tilde {\rm G}\to {\rm G}$ denote the simply connected cover of ${\rm G}$.
By \cite[Lem.~6.3]{GGN1}, the closure $\overline{{\rm X}(O_{(V_K\backslash S)\cup S'})}$ in $X_{S'}$
is a union of finitely many open orbits of $p(\tilde G_{S'})$. This proves the lemma.
\end{proof}

\vspace{0.2cm}

\subsection*{Proof of Theorem \ref{th:main2}}
The assumption (\ref{eq:main4}) implies that ${\rm X}(K)\cap D$ is infinite.
In particular, it follows that the set $X_{S,S'}\cap {\rm X}(K)$, embedded in $X_{S'}$, is not discrete,
and by Lemma \ref{l:discrete}, 
$$
\overline{X_{S,S'}\cap {\rm X}(K)}=\overline{{\rm X}(O_{(V_K\backslash S)\cup S'} )},
$$
is open in $X_{S'}$. Moreover, the closure 
$\overline{{\rm X}(O_{(V_K\backslash S)\cup S'})}$ in $X_{S'}$
is a union of finitely many open orbits of $p(\tilde G_{S'})$.
Therefore, it suffices to show that 
for $x^0\in {\rm X}(O_{(V_K\backslash S)\cup S'})$, almost all points in $p(\tilde G_{S'})x^0$
are approximable. Moreover, it suffices to show that for every compact subset $\Omega$ of 
$p(\tilde G_{S'})$, almost all points in $\Omega x^0$ are approximable.
From now on we fix such $\Omega$ and $x^0$.

If $\psi_v(t)\nrightarrow 0$ as $t\to \infty$ for all $v\in S'$, then
the claim of the theorem follows from density.
Hence, we assume that $\psi_v(t)\rightarrow 0$ as $t\to \infty$ for at least one $v\in S'$.

As in the proof of Theorem \ref{th:main1}, we may assume, without loss of generality,
that $\hbox{Im}(\psi_v)\subset I_v\cap (0,\epsilon_0)$
with notation as in Proposition \ref{p:dual2}. 

We set 
$$
U^0=\prod_{v\in V_K^f\cap (S\backslash S')} (U_v\cap {\rm G}(O_v))\quad\hbox{and}\quad U=U_{V_K^f\backslash S}U^0.
$$
Since both $U_v$ and ${\rm G}(O_v)$ are open and compact in $G_v$, it follows
that the subgroup $U_v\cap {\rm G}(O_v)$ has finite index in $U_v$.
Hence, since $U_v={\rm G}(O_v)$ for almost all $v$,
$U^0$ is a finite index subgroup in $U_{V_K^f\cap (S\backslash S')}$.
Then $U$ is a finite index subgroup in $U_{V_K^f\backslash S'}$.
Recall that $G^U$ denotes the kernel of $U$-invariant automorphic characters of $G_{V_K}$.
It contains ${\rm G}(K)$ and has finite index in $G_{V_K}$ (see \cite[Lem.~4.4]{GGN1}).
We note that $p(\tilde G_{S'})\subset G^U$ because $\tilde G$ has no nontrivial automorphic characters
(see \cite[Lem.~4.1]{GGN1}).
We also fix a compact neighbourhood $U^\prime$ of identity in $G_{V_K^\infty\backslash S'}$ contained in $G^U$.
Then $UU^\prime$ is a neighbourhood of identity in $G_{V_K\backslash S'}$.
Let
$$
B_h := U_{V_K\backslash S}\{g \in G_{V_K \backslash S}~:~ \height(g) \le h\}U_{V_K\backslash S},
$$
and 
$$
B_h':=U^0(B_h\cap G^U).
$$
Since $U$, $U^\prime$, and $U_{V_K\backslash S}$ are compact,
there exists $c_1\ge 1$ such that
\begin{equation}\label{eq:c1}
\sup_{b\in (UU^\prime)^{-1}B_h'} \height(b)\le c_1\, h.
\end{equation}
There exists $c_2=c_2(x^0)\ge 1$ such that 
\begin{equation}\label{eq:c2}
\height(\gamma x^0)\le c_2\, \height(\gamma).
\end{equation}

To simply notation, we set 
\begin{equation}\label{eq:notation}
d:=\dim({\rm G}),\;\; \mathfrak{q}:=\mathfrak{q}_{V_K \backslash S}({\rm G}),\;\;
\mathfrak{a}:=\mathfrak{a}_{S,S'}({\rm X}),\;\; \mathfrak{a}_0:=\mathfrak{a}_{S}({\rm G}).
\end{equation}


Since the function $\psi_S$ is nonincreasing, we deduce from \eqref{eq:main4} that
$$
\sum_{n = 1}^{\infty} |{\rm X}(K)\cap D \cap\{2^{n-1}< \height \le 2^n\}|\cdot \psi_S(2^{n-1})^{\alpha} = \infty,
$$
and from the definition of $\mathfrak{a}$, we also get
\begin{equation}\label{newsum1_1}
\sum_{n = 1}^{\infty} 2^{(\mathfrak{a}+\delta)n}\psi_S(2^{n-1})^{\alpha} = \infty
\end{equation}
for every $\delta>0$. 
Since $0< \psi_S\le 1$, there exists $\alpha_0(\delta)\in [0,\infty]$ such that series (\ref{newsum1_1})
converges for all $\alpha>\alpha_0(\delta)$
and diverges for all $\alpha<\alpha_0(\delta)$. 
We fix $\alpha_0>\mathfrak{a}\mathfrak{a}_0^{-1}\mathfrak{q}d/2$ such that series \eqref{eq:main4} diverges.
Since divergence in \eqref{eq:main4} implies divergence in (\ref{newsum1_1}),
we have $\alpha_0(\delta)\ge \alpha_0$. 

Since $\psi_S$ is monotone, \eqref{newsum1_1} is equivalent to
\begin{equation}\label{newsum1_1_1}
\sum_{n = 1}^{\infty} 2^{(\mathfrak{a}+\delta)n}\psi_S(c2^{n-1})^{\alpha} = \infty
\end{equation}
with any $c>0$. We choose $c=c_0 c_1 c_2/2$ where $c_0$ is as in (\ref{eq:low2})
and $c_1,c_2$ as in \eqref{eq:c1}, \eqref{eq:c2}.

As in the proof of Theorem \ref{th:main1}, we make a reduction to the case when
$\alpha_0(\delta)<\infty$. Let $\alpha_0>\mathfrak{a}\mathfrak{a}_0^{-1}\mathfrak{q}d/2$ be such that series (\ref{newsum1_1_1})
diverges. We define $R_n$ as in the proof of Theorem \ref{th:main1}. 
Then $R_n\le \psi_S(c\,2^{n-1})$,  $R_n\ll 2^{-\kappa n}$ with some $\kappa>0$, and 
\begin{equation}\label{newsum1111}
\sum_{n = 1}^{\infty} 2^{(\mathfrak{a}+\delta)n}R_n^{\alpha} = \infty
\end{equation}
for $\alpha=\alpha_0$. 
Since $R_n\ll 2^{-\kappa n}$, series (\ref{newsum1111}) converges for sufficiently large $\alpha$.
There exists $\alpha_1(\delta)\in [\alpha_0,\infty)$ such that series (\ref{newsum1111}) converges
for all $\alpha>\alpha_1(\delta)$ and diverges for all $\alpha<\alpha_1(\delta)$.
Since $\psi_S(t)=\prod_{v\in S'} \psi_v(t)^{r_v}$ and $R_n\le \psi_S(c\,2^{n-1})$,
we may write 
$$
R_n=\prod_{v\in S'} (\epsilon_v^{(n)})^{r_v},
$$
where $\epsilon_v^{(n)}\le \psi_v(c\,2^{n-1})$ for $v\in S'$.

Let $\beta'_n$ be the Haar-uniform probability measure supported on $B_n'$. 
By Theorem \ref{c:mean}, the averages along $\beta'_n$ satisfy the mean ergodic theorem:
for every $\phi\in L^2(\Upsilon)$ such that $\supp(\phi)\subset G^U/{\rm G}(K)$,
\begin{align}\label{eq:mean}
&\left\|\pi_{(V_K\backslash S)\cup (V^f_K\backslash S')}(\beta_n')\phi -\left(\int_\Upsilon \phi\,
    d\mu\right)\xi_U \right\|_2 \nonumber \\
\ll_{\delta}\; &
m_{V_K\backslash S}(B_{2^n}\cap G^U)^{-1/\mathfrak{q}+\delta}\|\phi\|_2 
\end{align}
for every $\delta>0$, where $\xi_U$ is the function on $\Upsilon$
such that $\xi_U=|G_{V_K}:G^{U}|$ on $G^{U}/{\rm G}(K)\subset \Upsilon$
and $\xi_U=0$ otherwise.

Let $\Phi_{n}$ be a family of measurable subsets of $\Upsilon$ defined by Proposition \ref{p:dual2}
with $\epsilon_v=\epsilon_v^{(n)}$, $v\in S'$. We set 
$$
\phi_n:=c_n 1_{\Phi_n}\quad\hbox{with $c_n=2^{(\mathfrak{a}+\delta)n}R_n^{\alpha-d}$}.
$$

By construction of the sets $\Phi_n$ in Proposition \ref{p:dual2},
$\Phi_n$ is a neighbourhood of the identity coset in $\Upsilon$
with size determined by $\epsilon_v^{(n)}\le \psi_v(c\,2^{n-1})$ (see the proof of \cite[Prop.~5.5]{GGN1}).
Since the divergence condition (\ref{eq:main4}) 
is stable under rescaling of the functions $\psi_v$,
we may arrange that 
$$
\supp(\phi_n)=\Phi_n\subset G^U/{\rm G}(K).
$$
By \eqref{eq:low1},
\begin{equation}\label{eq:l2}
\|\phi_n\|_2=c_n\mu(\Phi_n)^{1/2}\ll 2^{(\mathfrak{a}+\delta)n}R_n^{\alpha-d/2}.
\end{equation}
Let
$$
F_{k} := \sum_{n \geq k} \left|\pi_{(V_K\backslash S)\cup (V_K^f\backslash S')}(\beta'_n)\phi_k - \left(\int_{\Upsilon} \phi_k~d\mu \right)\xi_U \right|.
$$
As in the proof of Theorem \ref{th:main1}, we claim that $F_k$ is $L^2$-integrable for sufficiently small 
$\delta>0$ and
$\alpha<\alpha_1(\delta)$ sufficiently close to $\alpha_1(\delta)$.
By \eqref{eq:mean} and \eqref{eq:l2}, for every $\delta>0$,
\begin{align*}
\|F_{k} \|_{2}    &\ll_{\delta}
\sum_{n \geq k} m_{V_K\backslash S}(B_{2^n}\cap G^U)^{-1/\mathfrak{q}+\delta}\|\phi_n\|_2\\
&\ll \sum_{n \geq k} m_{V_K\backslash S}(B_{2^n}\cap G^U)^{-1/\mathfrak{q}+\delta}2^{(\mathfrak{a}+\delta)n}R_n^{\alpha-d/2}.
\end{align*}
Moreover, by \cite[Lem.~6.1-6.2]{GGN1}, for every $\delta>0$ and sufficiently large $n$,
$$
m_{V_K\backslash S}(B_{2^n}\cap G^U)\gg m_{V_K\backslash S}(B_{2^n})\gg_\delta 2^{(\mathfrak{a}_0-\delta)n}.
$$
Therefore, for sufficiently large $k$,
\begin{align*}
\|F_{k} \|_{2}  &\ll_\delta \sum_{n\ge k} 2^{(\mathfrak{a}-\mathfrak{a}_0/\mathfrak{q}+\theta)n}R_n^{\alpha-d/2}\\
&= \sum_{n\ge k} 2^{(\mathfrak{a}-\mathfrak{a}_0/\mathfrak{q}+2\theta)n}R_n^{\alpha-d/2}\cdot 2^{-\theta n},
\end{align*}
where $\theta=\theta(\delta)$ satisfies $\theta(\delta)\to 0^+$ as $\delta\to 0^+$.
Now we apply H\"older's  inequality with the exponents
$$
r=\mathfrak{a}/(\mathfrak{a}-\mathfrak{a}_0/\mathfrak{q}+2\theta)\quad\hbox{and}\quad \bar r=(1-1/r)^{-1}.
$$
When $\delta$ is sufficiently small, $r>1$.
This gives
\begin{align}\label{eq:last1}
\|F_{k} \|_{2}  &\ll_\delta \left(\sum_{n\ge k} 2^{\mathfrak{a}n}R_n^{r(\alpha-d/2)}\right)^{1/r}
\cdot \left(\sum_{n\ge k} 2^{-\theta\bar rn}\right)^{1/\bar r}.
\end{align}
Since $\alpha_0>\mathfrak{a}\mathfrak{a}_0^{-1}\mathfrak{q}d/2$, it is easy to check that 
$$
\frac{\mathfrak{a}(\alpha_0-d/2)}{\mathfrak{a}-\mathfrak{a}_0/\mathfrak{q}}>\alpha_0.
$$
Moreover, it follows that for all sufficiently small $\delta>0$ and all $\alpha\ge \alpha_0$,
$$
\frac{\mathfrak{a}(\alpha-d/2)}{\mathfrak{a}-\mathfrak{a}_0/\mathfrak{q}+2\theta}>\alpha.
$$
Since $\alpha_1(\delta)\ge \alpha_0$, 
it follows that for $\alpha<\alpha_1(\delta)$ sufficiently close to $\alpha_1(\delta)$ and
for sufficiently small $\delta>0$,
$$
r(\alpha-d/2)=\frac{\mathfrak{a}(\alpha-d/2)}{\mathfrak{a}-\mathfrak{a}_0/\mathfrak{q}+2\theta}>\alpha_1(\delta).
$$
Hence, by the definition of $\alpha_1(\delta)$,
the series in \eqref{eq:last1} converges, which completes the proof that $F_k$ is $L^2$-integrable. 

We fix $\delta>0$ and $\alpha<\alpha_1(\delta)$ such that $F_k$ is $L^2$-integrable.
Let
$$
\Upsilon_n: = \{ \varsigma\in G^U/{\rm G}(K):\, (B_{2^n}')^{-1} \varsigma \cap \Phi_n = \emptyset\}.
$$
By the definition of $F_k$, on the set $\cap_{n\ge k}\Upsilon_n$,
$$
F_k=|G_{V_K}:G^U|\sum_{n\ge k}\int_\Upsilon \phi_n\, d\mu=|G_{V_K}:G^U|\sum_{n\ge k} c_n \mu(\Phi_n).
$$
Since by Proposition \ref{p:dual2}, $\mu(\Phi_n)\gg R_n^d$ and $\alpha<\alpha_1(\delta)$,
we conclude that
$$
F_k\gg \sum_{n\ge k} 2^{(\mathfrak{a}+\delta)n}\psi_S(c 2^{n-1})^{\alpha}=\infty
$$
on the set $\cap_{n\ge k}\Upsilon_n$. In particular, this shows that $\mu(\cap_{n\ge k} \Upsilon_n)=0$
and $\Upsilon_\infty:=\liminf(\Upsilon_n)$ also has measure zero.
Let $\widetilde\Upsilon_\infty$ be the preimage of $\Upsilon_\infty$ in $G^U$.
Then $m(\widetilde\Upsilon_\infty)=0$. Let
$$
\Omega'=\{g\in \Omega:\,\, \exists y\in UU^\prime:\, (y,g^{-1})\notin\widetilde\Upsilon_\infty\}.
$$
Then since
$$
(UU^\prime\times (\Omega\backslash \Omega')^{-1})\subset \widetilde\Upsilon_\infty,
$$
and $UU^\prime$ has positive measure in $G_{V_K\backslash S'}$, the set $\Omega\backslash \Omega'$ has measure
zero. Then $\Omega'x^0$ has full measure in $\Omega x^0$.

Finally, we show that almost every element of $\Omega'x^0$
belongs to the set $\mathcal{W}(X_{S'}, X_{S,S'}\cap {\rm X}(K),\Psi)$.
For $g\in \Omega'$, we set $\varsigma:=(e,g^{-1}){\rm G}(K)$.
There exists $y\in UU^\prime$ such that 
$$
y\varsigma\notin \Upsilon_\infty.
$$
This implies that for infinitely many $n$, $y\varsigma\in \Upsilon_n$, i.e.,
$$
(y^{-1}B_{2^n}')^{-1}\varsigma\cap \Phi_n\ne \emptyset.
$$
Then it follows from Proposition \ref{p:dual2} that for infinitely many $n$,
there exists $\gamma_n\in {\rm G}(O_{(V_K\backslash S)\cup S'})$ such that
$$
\height(\gamma_n)\le c_0 \sup_{b\in (UU^\prime)^{-1}B_{2^n}'} \height(b)\le c_0 c_1 2^n,
$$
and for $x=gx^0$ and $z_n=\gamma_n x^0$,
$$
\|x_v-z_n\|_v \le \epsilon_v^{(n)}\le \psi_v(c 2^{n-1})\quad \hbox{for all $v\in S'$.}
$$
We have  $z_n \in {\rm X}(O_{(V_K\backslash S)\cup S'})$ and $\height(z_n)\le c_2\height(\gamma_n)\le c
2^{n-1}$. Hence, since $\psi_v$ is monotone, we conclude that
$$
\|x_v-z_n\|_v \le \psi_v(\height(z_n))\quad \hbox{for all $v\in S'$}.
$$
Recall that $\psi_v(t)\rightarrow 0$ as $t\to \infty$ for at least one $v\in S'$.
If $x_v\notin {\rm X}(O_{(V_K\backslash S)\cup S'})$, then it follows that the set $\{z_n\}$
is infinite. Therefore, almost every element of $\Omega'x^0$ is in 
$\mathcal{W}(X_{S'}, X_{S,S'}\cap {\rm X}(K),\Psi)$.
This completes the proof of the theorem.

\section{Hausdorff dimension}\label{sec:h}

We start by recalling the notion of Hausdorff measure and dimension.
Let $(M,\text{dist})$ be a locally compact separable metric space. 
The $s$-Hausdorff measure $\cH^{s}$ is a Borel measure on $M$ defined by
$$
\cH^{s} (E) := \lim_{\rho \to 0^+} \cH^{s}_{\rho}(E),
$$
where
$$
\cH^{s}_{\rho}(E) := \inf \sum_{i} r(B_i)^s,
$$
and the infimum is taken over all countable covers of $E$ by closed balls $B_i$
such that each $B_i$ has radius at most $\rho$, and 
$r(B_i)$ denotes the radius of $B_i$.
The Hausdorff dimension of the set $E$ is defined by
$$
\dim(E):=\sup\{s:\, \cH^s(E)=\infty\}=\inf\{s:\, \cH^s(E)<\infty\}.
$$
We assume that for some $d,r_0>0$,
\begin{equation}\label{eq:hd}
r^d \ll \cH^{d}(B(x,r)) \ll r^d
\end{equation}
uniformly over all closed balls $B(x,r)$ with $r\leq r_0$.  
Then for every nontrivial closed ball $B$ and $s<d$, we have $\cH^{s}(B)=\infty$.
 
The following Mass Transfer Principle was proved in \cite{BV}:

\begin{Theorem}[\cite{BV}, Th.~3]\label{th:BV}
Let $\{B(x_i,r_i)\}_{i \in \N}$ be a
sequence of closed balls in $M$ with $r_i \to 0$ as $i \to \infty$. 
Suppose that for some $s\in (0,d)$ and every closed ball $C$ in $M$,
\begin{equation}\label{eq:hdd}
\mathcal{H}^{d}\left(C \cap \limsup B(x_i,r_i^{s/d})\right) = \mathcal{H}^{d}(C).
\end{equation}
Then for a closed ball $B$ in $M$,
\begin{equation}
\mathcal{H}^{s}\left(B \cap \limsup B(x_i, r_i)\right) = \mathcal{H}^{s}(B).
\end{equation}
\end{Theorem}

The proof of Theorem \ref{th:BV} is based on construction of a Cantor-like set in
$B \cap \limsup_{i\in \N} B(x_i, r_i)$ and a suitable measure supported on this set
with large dimension. The same proof still applies provided that 
\begin{enumerate}
\item[(i)] (\ref{eq:hd}) holds for all closed balls $B(x,r)$ with $r\le r_0$ such that $B\cap B(x,r)\ne\emptyset$,
\item[(ii)] (\ref{eq:hdd}) holds for all closed balls $C$ contained in the ball $B$.
\end{enumerate}

\vspace{0.2cm}

\subsection*{Proof of Theorem \ref{th:main3}}
In the proof we use notation as in \eqref{eq:notation},
and without loss of generality, we may assume that $\psi(t)\to 0$ as $t\to\infty$.

The assumption (\ref{eq:main4}) implies that 
the set $X_{S,S'}\cap {\rm X}(K)$, embedded in $X_{S'}$, is not discrete,
and by Lemma \ref{l:discrete}, its closure 
$\overline{X_{S,S'}\cap {\rm X}(K)}$ is open in $X_{S'}$. 

We consider the space $X_{S'}$ with the metric which is the product of local metrics \eqref{eq:dist}.
We cover ${\rm X}$ with a collection of Zariski open subsets ${\rm U}$ such that each ${\rm U}$
supports a non-vanishing regular differential form of top degree. 
Since the sets $U_{S'}$ form an open cover of $X_{S'}$, it is sufficient to show that
$$
\dim \left(B_0\cap \mathcal{W}(X_{S,S'}, X_{S,S'}\cap {\rm X}(K),\psi)\right)\ge \frac{2\mathfrak{a}_0}{\mathfrak{a}\mathfrak{q}}\alpha.
$$
for all nontrivial closed balls $B_0$ in $X_{S'}$ such that 
$B_0\subset U_{S'}\cap \overline{X_{S,S'}\cap {\rm X}(K)}$ for some $U_{S'}$. 

Let 
$$
\rho:=\text{dist}\left(B_0,(U_{S'}\cap \overline{X_{S,S'}\cap {\rm X}(K)})^c\right)>0,
$$
and 
$$
\tilde B_0:=\{x\in X_{S'}:\, \text{dist}(x, B_0)\le \rho/2\}\subset U_{S'}\cap \overline{X_{S,S'}\cap {\rm X}(K)}.
$$
Then every closed ball $B(x,r)$ such that $r\le \rho/4$ and
$B(x,r)\cap B_0\ne \emptyset$ satisfies $B(x,r)\subset \tilde B_0$.

Let $\lambda_{S'}$ be the measure on $U_{S'}$ defined by the nowhere-zero differential form.
Different choices of differential forms lead to equivalent measures.
Since ${\rm X}$ is a homogeneous variety, it is nonsingular, and computing in local coordinates,
we obtain that for all closed balls $B(x,r)\subset U_{S'}$,
$$
r^{d}\ll \lambda_{S'}(B(x,r))\ll r^{d},
$$
where $d:=(\sum_{v\in S'} r_v)\dim({\rm X})$.
This estimate is uniform over all closed balls $B(x,r)$ with $x\in \tilde B_0$ and bounded $r$.
Therefore, we conclude that
\begin{equation}\label{eq:hhd}
\mathcal{H}^{d} \ll \lambda_{S'} \ll \mathcal{H}^{d}
\end{equation}
uniformly over Borel subsets of $\tilde B_0$. In particular,
it follows that property (i) (stated after Theorem \ref{th:BV}) holds.

We apply Theorem \ref{th:BV} to 
the collection of closed ball $B(z,\psi(\height(z)))$ with $z\in X_{S,S'}\cap {\rm X}(K)$.
Choose any positive $s<\frac{2\mathfrak{a}_0}{\mathfrak{a}\mathfrak{q}}\alpha$ and set $\tilde\psi :=\psi^{s/d}$.
We have 
$$
\mathcal{W}(X_{S,S'}, X_{S,S'}\cap {\rm X}(K),\tilde \psi)=\limsup B(z,\tilde \psi(\height(z))).
$$
According to our assumption, for $\beta:=\frac{d}{s}\alpha>\frac{\mathfrak{a}\mathfrak{q}}{2\mathfrak{a}_0}d$,
$$
\sum_{z\in {\rm X}(K)\cap D} \tilde\psi_{S'}(\height(z))^\beta=\infty.
$$
Hence, by Theorem \ref{th:main2}, the set 
$\mathcal{W}(X_{S,S'}, X_{S,S'}\cap {\rm X}(K),\tilde \psi)$ has full measure in $X_{S'}$,
and in particular, it follows from \eqref{eq:hhd} that for every closed ball $C$ contained in $B_0$,
$$
\mathcal{H}^{d}\left(C \cap \limsup  B(z,\tilde\psi(\height(z)))\right) = \mathcal{H}^{d}(C).
$$
This verifies property (ii) (stated after Theorem \ref{th:BV}), and Theorem \ref{th:BV} now implies that
$$
\mathcal{H}^{s}\left(B_0 \cap \limsup B(z,\psi(\height(z)))\right) = \mathcal{H}^{s}(B_0)=\infty
$$
for every $s<\frac{2\mathfrak{a}_0}{\mathfrak{a}\mathfrak{q}}\alpha$. This proves that 
$$
\dim \left(B_0\cap \mathcal{W}(X_{S,S'}, X_{S,S'}\cap {\rm X}(K),\psi)\right)\ge \frac{2\mathfrak{a}_0}{\mathfrak{a}\mathfrak{q}}\alpha,
$$
as required.

\end{document}